\documentclass[11pt]{amsart}
\usepackage[left=1.5in, right=1.5in]{geometry} 
\usepackage{
 amsmath, 
 amsxtra, 
 amsthm, 
 amssymb, 
 etex, 
 mathrsfs, 
 mathtools, 
 tikz-cd, 
 bbm,
 xr,
 comment,
 enumitem}
\usepackage{lipsum}
\usepackage{tabularx}
\usepackage{cancel}
\usepackage[all]{xy}
\usepackage{hyperref}
\usepackage{url}
\usepackage{tipa}
\usepackage[normalem]{ulem}
\usepackage{dsfont}
\usepackage{colonequals}
\usepackage{stmaryrd}
\usepackage{color}
\usepackage{upgreek}
\usepackage[colorinlistoftodos,prependcaption,textsize=tiny]{todonotes}

\newtheorem{theorem}{Theorem}[section]
\newtheorem{lemma}[theorem]{Lemma}

\newtheorem{proposition}[theorem]{Proposition}

\numberwithin{equation}{section}

\theoremstyle{remark}
\newtheorem{remark}[theorem]{Remark}

\usepackage[utf8]{inputenc}

\newcommand\EatDot[1]{}

\newcommand{\cyc}{{\mathrm{cyc}}}

\newcommand{\gp}{{\mathfrak{p}}}

\newcommand{\Gal}{{\mathrm{Gal}}}

\newcommand{\Sel}{{\mathrm{Sel}}}

\newcommand{\Q}{{\mathbb Q}}
\newcommand{\Z}{{\mathbb Z}}

\newcommand{\OO}{{\mathcal O}}

\DeclareMathOperator{\coker}{coker}

\newcommand{\col}{\mathrm{Col}}

\newcommand{\Qp}{\mathbb{Q}_p}
\newcommand{\Zp}{\mathbb{Z}_p}

\newcommand{\HIw}{H^1_{\mathrm{Iw}}}
\newcommand{\BK}{\mathrm{BK}}

\definecolor{Green}{rgb}{0.0, 0.5, 0.0}

\newcommand{\T}{\mathcal{T}}
\newcommand{\A}{\mathcal{A}}

\title[~]{On characteristic power series of dual signed Selmer groups}

\author[J. Ray]{Jishnu Ray}
\address[Ray]{Harish Chandra Research Institute, A CI of Homi Bhabha National Institute,  Chhatnag Road, Jhunsi, Prayagraj (Allahabad) 211 019 India}
\email{jishnuray@hri.res.in}

\author[F. Sprung]{Florian Sprung}
\address[Sprung]{School of Mathematical and Statistical Sciences, Arizona State University Tempe, AZ 85287-1804, USA}
\email{florian.sprung@asu.edu}



\begin{document}

\maketitle
\begin{abstract}
We relate the cardinality of the $p$-primary part of the Bloch-Kato Selmer group over $\Q$ attached to a modular form at a non-ordinary prime $p$ to the constant term of the characteristic power series of the signed Selmer groups over the cyclotomic $\Z_p$-extension of $\Q$. This generalizes a result of Vigni and Longo in the ordinary case.  In the case of elliptic curves, such  results follow from earlier works by Greenberg, Kim, the second author, and Ahmed--Lim, covering both the ordinary and most of the supersingular case.
\end{abstract}

\section{Introduction}
The aim of this paper is to relate the size of a $p$-adic Selmer group attached to a modular form which is non-ordinary at $p$ to the constant term of the characteristic power series of the cyclotomic deformation of this Selmer group. This has been done in the case of elliptic curves, in  the ordinary and most of the supersingular (i.e. non-ordinary) case, and in the case of modular forms, but working with ordinary $p$. We review what is known in the next three subsections before describing our contribution.
\subsection{The case of elliptic curves (ordinary at $p$)} Let $E$ be an elliptic curve defined over a number field $F$ such that $E$ has good ordinary reduction at all the primes of $F$ above an odd prime $p$. Assume the $p$-primary Selmer group  $\Sel_p(E/F)$ is finite. Let $F_\cyc$ be the cyclotomic $\Z_p$-extension  of $F$. Then by a control theorem, $\Sel_p(E/F_{\rm{cyc}})$ is known to be cotorsion over the corresponding Iwasawa algebra of the Galois group $\Gamma_0:=\Gal(F_{\rm{cyc}}/F)$. Let $f _E(X) $ generate the characteristic ideal of the Pontryagin dual of $\Sel_p(E/F_{\rm{cyc}})$, considered as a power series by identifying $1+X$ with a topological generator of $\Gamma_0$. In \cite[Thm. 4.1]{Greenberg}, Greenberg showed that $f_E(0)\neq 0$ and 
\begin{equation}\label{eq:green}
f_E(0) \sim \frac{\#\Sel_p(E/F) \cdot \prod_{v \text{ bad }}c_v(E) \cdot \prod_{v \mid p} \#\big(\tilde{E}_v(\mathds{F}_v)_p\big)^2}{\#\big(E(F)_p\big)^2},
\end{equation}

where the symbol $\sim$ means that the two quantities differ by a $p$-adic unit. 
Here $c_v(E)$ is the Tamagawa number of $E$ at a prime $v$ of bad reduction, $\mathds{F}_v$ is the residue field of $F$ at $v$, $\tilde{E}_v(\mathds{F}_v)_p$ is the $p$-torsion of the group of $\mathds{F}_v$-rational points of the reduction $\tilde{E}_v$ of $E$ at $v$ and $E(F)_p$ is the $p$-torsion subgroup of the Mordell-Weil group $E(F)$.  


The product of quantities in \eqref{eq:green} can be reinterpreted as an Euler characteristic, so that
$$f_E(0) \sim \chi(\Gamma_0,\Sel_p(E/F_{\rm{cyc}}));$$
see \cite[Lemma 4.2]{Greenberg} both for the statement and the definition of $\chi(\Gamma_0,\Sel_p(E/F_{\rm{cyc}}))$.

\subsection{The case of elliptic curves (supersingular and mixed reduction)}  In the supersingular case, the classical Selmer group 
$\Sel_p(E/F_{\rm{cyc}})$ is not cotorsion. However, there are remedies. Denote by $S _p$ the primes of $F$ above $p$.  Suppose for the moment that for each $v \in S_p$, $a_v=1+p-\#\tilde{E}_v(\mathds{F}_p)=0 $, i.e. we are in a subcase of supersingular reduction. {Following the intuition that the cotorsion of the Selmer group fails because there are too many points in the local condition that defines it,} Kobayashi constructed signed (plus/minus) Selmer groups $\Sel_p^\pm(E/F_{\rm{cyc}})$ {using ``half'' of the local points}, and proved that when $F=\Q$, they are cotorsion over the corresponding Iwasawa algebra \cite{kobayashi03}.   
Suppose $p$ splits completely in $F$  and is totally ramified in $F_{\rm{cyc}}$.  The signed Selmer groups $\Sel_p^\pm(E/F_{\rm{cyc}})$ are conjectured to be cotorsion over the Iwasawa algebra of $\Gamma_0$. When this is the case, let $f^{\pm}$ be  generators of the characteristic power series of their Pontryagin duals.

When  $\Sel_p(E/F)$ is finite, one can use a control theorem to prove $\Sel_p^\pm(E/F_{\rm{cyc}})$ are cotorsion. Kim then showed that
$$f^{\pm}(0) \sim \#\Sel_p(E/F) \cdot \prod_\ell c_\ell.$$
Here $\ell$ runs over every prime and the Tamagawa factor $c_\ell=[E(\Q_\ell): E^0(\Q_\ell)]$. See \cite[Thm. 1.2 or Corollary 3.15 and the explanations in its proof]{Kim13} for the above assertions.

If one drops the assumption $a_v=0$, the second author has generalized Kobayashi's construction to construct a pair of chromatic Selmer groups $\Sel^\sharp$ and $\Sel^\flat$ \cite{sprung12}, and given an analogous formula of Kim's, all assuming that  $F=\Q$, see \cite[Lemmas 4.4, 4.5, 4.8]{sprung16}.

Returning to the case $a_v=0$, in Kim's results, the local condition appearing in the definition of $\Sel^+_p(E/F_{\rm{cyc}})$ is the same (``plus condition" of Kobayashi) for each $v$. Analogously, the local condition in $\Sel^-_p(E/F_{\rm{cyc}})$ is the same ``minus condition'' for each $v$. However, one may consider \textit{mixed} signed Selmer groups, where the local conditions may be different for different $v$. In this case, Kim's Euler characteristic formula has been generalized by Ahmed and Lim under certain additional hypotheses (see \cite[Thm. 1.1]{AhmedLim}).
\subsection{The case of modular forms (ordinary at $p$)}
Suppose $f(q)=\sum_{n \geq 1}a_n(f)q^n$ is a newform of even weight $k \geq 4$ and level $\Gamma_0(N)$. Let $\mathds{Q}_f=\Q(a_n(f) \mid n \geq 1)$ be the Hecke field of $f$. Suppose $ p \nmid 2N$ and fix a prime $\gp$ of $\mathds{Q}_f$ above $p$. Assume that $a_p(f)$ is a $\gp$-adic unit (i.e. $\gp$ is ordinary for $f$) and $a_p(f) \not\equiv 1 \text{ } (\text{mod } \gp).$ Following the notations of Longo and Vigni, we write $E$ for the completion of $\mathds{Q}_f$ at $\gp$ and $\OO_E$ for the valuation ring of $E$. Let $V$ be the self-dual twist of the representation $V_{f, \gp}$ of $G_\Q$ attached to $f$ by Deligne; hence $V=V_{f,\gp}(k/2)$. Choose a $G_\Q$-stable $\OO_E$-lattice $T$ of $V$ and set $A_f=V/T$. Let $\Sigma$ be a finite set of primes of $\Q$ containing $p$, primes dividing the level $N$ and the archimedean prime. 
Suppose that the Bloch-Kato Selmer group $\Sel_{\BK}(A_f/\Q)$ is finite. Then Greenberg's Selmer group $\Sel_{\mathrm{Gr}}(A_f/\Q_{\rm{cyc}})$ over the cyclotomic $\Z_p$-extension $\Q_{\rm{cyc}}$ of $\Q$ is cotorsion \cite[Prop. 4.1]{VigniLongo}. The reader may consult \cite[Sec. 3.4]{VigniLongo} for the definition of the Bloch-Kato and Greenberg Selmer groups. 
Let $\mathcal{F}$ be the  power series generating  the Pontryagin dual of $\Sel_{\mathrm{Gr}}(A_f/\Q_{\rm{cyc}})$. Vigni and Longo then showed the following Euler characteristic formula (see \cite[Thm. 1.1]{VigniLongo}).
\begin{equation}\label{VL}
\#(\OO_E/\mathcal{F}(0)\cdot \OO_E) = \# \Sel_{\BK}(A_f/\Q) \cdot \underset{v \in \Sigma, v \nmid p}{\prod} c_v(A_f).
\end{equation}
Here  $c_v(A_f):=[H^1_{\mathrm{ur}}(\Q_\ell,A_f):H^1_f(\Q_\ell,A_f)]$ (see \cite[Defn. 3.2]{VigniLongo}).

Vigni and Longo were actually able to deduce this Euler characteristic result for general ordinary representations satisfying additional hypotheses (see assumption 2.1 of \cite{VigniLongo}).

\subsection{The case of modular forms (non-ordinary at $p$)} This is the case missing from the literature and this is precisely what this article addresses. In this case, using the theory of Wach modules, Lei, Loeffler, and Zerbes constructed signed Selmer groups over $\Q_{\rm{cyc}}$ (see \cite{lei11compositio},  \cite{leiloefflerzerbes10},  \cite{leiloefflerzerbes11}). {Again, the issue with the classical Selmer group is that the local condition is too big. Instead of considering ``half'' of the local points as Kobayashi did, the authors are able to cut out two appropriate subspaces of the local condition that come from kernels of certain maps, the Coleman maps. This is an analogue and generalization of \cite{sprung12} for the case of elliptic curves.} These signed Selmer groups are generally conjectured to be cotorsion and this is known to hold in a large number of cases.  Many arithmetic properties of the signed Selmer groups have recently been proved by Hatley and Lei \cite{HatleyLei}. Our main goal in this article is to generalize \eqref{VL} for the characteristic power series obtained from the signed Selmer groups over $\Q_{\rm{cyc}}$. 

{To produce results analogous to those of Longo and Vigni work in the non-ordinary setting, we need the signed Selmer groups to behave as well as the Selmer groups in the ordinary case, but this is not guaranteed. Hatley and Lei in \cite{HatleyLei} considered families of twists of these signed Selmer groups, and showed that appropriate twists result in well-behaved signed Selmer groups. Our task is thus to work with these (infinitely many) twists of signed Selmer groups $\Sel_1(A(s)/\Q_{\rm{cyc}})$ and  $\Sel_2(A(s)/\Q_{\rm{cyc}})$ coming from twists $A(s)$ of $A_f$, where $s$ is an integer, and hope that they give us desirable arithmetic information.} For these $\Sel_i(A(s)/\Q_{\rm{cyc}})$, the local conditions at the prime $p$ are defined as the Tate-local orthogonal complement of the kernel of signed Coleman maps, which the interested reader can study in section \ref{sec:def} and the beginning of section \ref{main:body}. The arithmetic properties of such local conditions at the prime $p$ studied  by Hatley and Lei \cite{HatleyLei} need to be developed further for our purposes. In particular, we prove a ``Control Lemma at the prime $p$'' (Lemma \ref{a}), equating local cohomology groups of the form $H^1(\Qp, -)$ to $H^1(\Q_{\cyc,p}, -)^{\Gamma_0}$ sitting in an exact sequence. When the $\Sel_i(A(s)/\Q_{\rm{cyc}})$ are cotorsion,  let $f_{i,s}$ generate their characteristic ideals. Under some mild hypotheses made in \cite{HatleyLei} and some hypotheses on $s$, our main theorem is (see theorem \ref{Thm:main1})
$$\#\left(\frac{\OO_E}{f_{i,s}(0)\OO_E}\right)=\#\chi(\Gamma_0,\Sel_i(A(s)/\Q_{\rm{cyc}}))=\#\Sel_{\BK}(A(s)/\Q) \cdot  \prod_{\ell\in \Sigma, \\ \ell \nmid p} c_\ell(A(s)),$$
for $i=1$ or $i=2$, where $c_\ell(A(s)):=[H^1_{\mathrm{ur}}(\Q_\ell,A(s)):H^1_f(\Q_\ell,A(s))]$ is the $p$-part of the \textit{Tamagawa number} of $A(s)$ at $\ell$. Here, $\Sigma$ is still the same set of primes we chose in the previous subsection.

One comment about the hypotheses on $s$: The first equality holds for all but finitely many, while for the second, we need $s$ to be one of $0,\cdots, k-2$.

The astute reader might have noticed that the Tamagawa numbers are over the primes not diving $p$ in both our formula and in that of Longo--Vigni's, in contrast to the elliptic curve case. This is because only over those places the local conditions for the (signed) Selmer groups coincide with Greenberg's local conditions (and with the usual Bloch-Kato local conditions, again because we are away from $p$, see section \ref{sec: BKloc}). The recent techniques of Vigni and Longo \cite{VigniLongo} developed to deal with primes not dividing $p$ in the ordinary case are sufficiently general to come to the rescue in our setting, the non-ordinary case, as well.

We use this to perform a calculation of the Euler characteristic, and we are then in a position to deduce our result relating the characteristic power series of the Pontryagin dual of the signed Selmer groups with that of the cardinality of Bloch-Kato Selmer group over $\Q$ and the Tamagawa factors. 
\subsection{Outlook.}

Ponsinet obtained a similar formula for the constant terms of the characteristic power series for the dual signed Selmer groups of abelian varieties with supersingular reduction at $p$ such that the Hodge-Tate weights lie in $[0,1]$ (see \cite[Corollary 2.10]{Ponsinet_abelian}), under various other additional hypotheses (see \cite[Sections 1.2 and  1.3]{Ponsinet_abelian}). Proving such a result in a more general setting is currently underway. 

\section*{Acknowledgements and dedication}
We would like to thank Jeffrey Hatley, Antonio Lei and Meng-Fai Lim   for answering many questions and for their  comments on an earlier draft of this article. We also thank Matteo Longo and Stefano Vigni for their encouragement. The first author is supported by the Inspire research grant. The second author is supported by a Simons grant and an NSF grant. We also thank the referee for suggesting several improvements.

Sujatha started her work in Iwasawa theory by  computing an Euler characteristic for abelian varieties (in a noncommutative setting) back in the year 1999 \cite{CoatesSujatha1999}. It is a pleasure to dedicate this paper to her;  she introduced Iwasawa theory both to the first author and -- through the two books co-authored with John Coates \cite{CoatesSujathaotherbook,CoatesSujatha_book} -- to the second author of this article.

\section{Preliminaries}

\subsection{Notation} \label{sec:notation}
Several results in our paper rely on the results of \cite{HatleyLei}. In order to make it easier for the reader to jump between papers, we follow their notation (which in turn follows mostly that of \cite{Greenberg}) as much as possible. Thus, let $p$ be a fixed odd prime. The letter $\ell$ also denotes a prime. Let $\Q_{p,n}=\Q_p(\mu_{p^n})$, $\Q_\infty=\Q(\mu_{p^\infty})$ and $\Gamma=\Gal(\Q_\infty/\Q)\cong\Gal(\Q_p(\mu_{p^\infty})/\Q_p) \cong\Gamma_0 \times \Delta$, where $\Gamma_0\cong \Z_p$ and $\Delta \cong \Z/(p-1)\Z$. Let $\kappa$ and $\omega$ be the restriction of the cyclotomic character $\chi$ to $\Gamma_0$ and $\Delta$. Then $\Q_{\rm{cyc}}=\Q_\infty^{\Delta}$ is the  cyclotomic $\Z_p$-extension of $\Q$. In any subfield of $\Q_\infty$, we write $p$ for its unique prime above $p$, since $p$ is totally ramified in $\Q_\infty$.

Let $f=\sum a_n(f)q^n$ be a normalized new cuspidal eigenform of even weight $k \geq 2$, level $N$ and nebentypus $\varepsilon$. We assume that $p \nmid N$, $f$ is non-ordinary at $p$ and $a_n(f)$ is defined over a totally real field for all $n$. Let $E$ be a finite extension of $\Q_p$ containing $a_n(f)$ for all $n$, $\varpi$ a uniformizer of $E$, and $V_f$ the $E$-linear Galois representation attached to $f$ constructed by Deligne \cite{Deligne1971}. Let $T_f$ be the canonical $G_\Q$-stable $\OO_E$-lattice in $V_f$ defined by Kato \cite[Sec. 8.3]{Kato}. Let $A_f=V_f/T_f(1)$. Note that $V_f$ has Hodge-Tate weights $\{0,1-k\}$, where the convention is that the Hodge-Tate weight of the cyclotomic character is $1$. Let $\T_f$ be the Tate twist $T_f(k-1)$ which has Hodge-Tate weights $\{0,k-1\}$. 

We let $\Lambda:=\OO_E[[\Gamma_0]]$ and let $\Lambda(\Gamma):=\OO_E[\Delta][[X]]$, the Iwasawa algebra of $\Gamma$. There are certain $\Lambda(\Gamma)$-morphisms called Coleman maps for each $i=1,2$ constructed using the theory of Wach modules, $$\col_{f,i}:\HIw(\Q_p,\T_f) \rightarrow \Lambda(\Gamma)$$  (see e.g.\cite[p. 1264 after eq. (2.1)]{HatleyLei}). Here, $\HIw(\Q_p,\T_f):=\varprojlim H^1(\Q_{p,n},\T_f)$.
For any finite extension $K/\Q_p$ contained in $\Q_{\infty,p}$, by Tate-duality, there exists the Tate pairing
\begin{equation}\label{TP}
H^1(K,\T_f) \times H^1(K,A_f) \rightarrow \Q_p/\Z_p.
\end{equation}
Upon taking direct and inverse limits, we obtain a pairing 
\begin{equation}\label{Iwpairing}
\HIw(\Q_p,\T_f) \times H^1(\Q_{\infty,p},A_f) \rightarrow \Q_p/\Z_p.
\end{equation}

{Let $H^1_i(K,A_f)$ be the  orthogonal complement (under the Tate-pairing \eqref{TP}) of the image of $\ker \col_{f,i}$ under the natural map $\HIw(\Q_p,\T_f) \rightarrow H^1(K_v,\T_f).$
Let $H_i^1(\Q_{\infty,p},A_f)$ be the direct limit $\varinjlim H^1_i(K,A_f)$. The group $H_i^1(\Q_{\infty,p},A_f)$ can also be identified with the orthogonal complement of $\ker \col_{f,i}$ under the  pairing \eqref{Iwpairing}. In \cite[Remark 2.5]{HatleyLei}, Hatley--Lei showed the following compatibility relation, which is crucial when analyzing the control diagram for the  local conditions at $p$ (see Lemma \ref{a2}):
\begin{equation}\label{Comp_local_p}
H^1_i(\Q_{\infty,p},A_f)^{\Gal(\Q_{\infty,p}/K)}=H^1_i(K,A_f).
\end{equation}
Following \cite[Section 1.6]{Ponsinet_abelian} or \cite[Section 2.1]{LP20}, we set 
\begin{equation}\label{cyc_loc_p}
H^1_i(\Q_{\cyc,p},A_f):=H^1_i(\Q_{\infty,p},A_f)^{\Delta}.
\end{equation}
} These are the local condition at $p$ which will be used in defining the signed Selmer groups below.
When $v$ is a prime of $\Q_\infty$ not above the prime $p$, the local condition is
\begin{equation}\label{unr}
H^1_i(\Q_{\infty,v},A_f):=H^1_{\mathrm{un}}(\Q_{\infty,v},A_f),
\end{equation}
which is the unramified subgroup of $H^1(\Q_{\infty,v},A_f)$.

\subsection{The signed Selmer groups}\label{sec:def} The signed Selmer groups over $\Q_\infty$ are defined as 

$$\Sel_i(A_f/\Q_\infty)=\ker\Big(H^1(\Q_\infty,A_f) \rightarrow \prod_vH^1_{/i}(\Q_{\infty,v},A_f)\Big),$$
where $v$ runs through all the places of $\Q_\infty$ and $H^1_{/i}(\Q_{\infty,v},A_f):=\frac{H^1(\Q_{\infty,v},A_f)}{H^1_i(\Q_{\infty,v},A_f)}.$

When $L$ is a subfield of $\Q_\infty$ and $v$ is a prime of $L$ not above $p$, we can define  $H^1_i(L_v,A_f)$ in the same way as in \eqref{unr}. 

Hence the signed Selmer groups are defined as 
$$\Sel_i(A_f/L):=\ker\Big(H^1(L,A_f) \rightarrow \prod_v H^1_{/i}(L_v,A_f)\Big).$$ 
{Using \eqref{cyc_loc_p},  we obtain $$\Sel_i(A_f/\Q_\cyc) \cong \Sel_i(A_f/\Q_\infty)^\Delta$$ as in \cite[p. 1267]{HatleyLei}.}
 Given an integer $s$, let $A_{f,s}=A_f \otimes \chi^s$, and $\T_{f,s}=\T_f \otimes \chi^{-s}$. These twists commute with cohomology, so that one can define twisted the signed Selmer groups $\Sel_i(A_{f,s}/\Q_\infty)$ and $\Sel_i(A_{f,s}/L)$ (cf. \cite[p. 1268 before Remark 2.6]{HatleyLei}). Recall that $\kappa$ (resp. $\omega$) is the restrictions of the cyclotomic character $\chi$ to $\Gamma_0$ (resp. $\Delta$).  
 For $i \in \{1,2\}$, the signed Selmer groups over the cyclotomic $\Z_p$-extension $\Q_{\rm{cyc}}$ are then given by
 $$\Sel_i(A_{f,s}/\Q_{\rm{cyc}}) \cong 
 \Sel_i(A_{f,s}/\Q_\infty)^\Delta \cong \Sel_i(A_f/\Q_\infty)^{\omega^{-s}} \otimes \kappa^s,$$
 as $\Lambda=\OO_E[[\Gamma_0]]$-modules,  see \cite[Remark 2.6]{HatleyLei}. 
 Here  $\Sel_i(A_f/\Q_\infty)^{\omega^{-s}}$ is the $\omega^{-s}$-isotypic component of $\Sel_i(A_f/\Q_\infty)$. If $\theta$ is a character of $\Delta$, then the $\theta$-isotypic component of $\Sel_i(A_{f,s}/\Q_\infty)$ is $\Sel_i(A_{f,s}/\Q_\infty)^\theta$ and we have the following isomorphisms of $\Lambda$-modules.
 $$\Sel_i(A_{f,s}/\Q_\infty)^\theta \cong \Sel_i(A_f/\Q_\infty)^{\theta \omega^{-s}} \otimes \kappa^s \cong \Sel_i(A_{f,s}(\theta^{-1})/\Q_{\rm{cyc}})$$
(See again \cite[a little further down in Rem. 2.6]{HatleyLei}).
 We record the following hypotheses also made in \cite{HatleyLei}.

 \noindent\textbf{(irred)}: The $G_\Q$-representation $T_f/\varpi T_f$ is irreducible.

 \noindent\textbf{(inv)}: For all $m \geq 0$, $A_f(m)^{G_{\Q_\infty,p}}=0$.
 
  \noindent\textbf{(tor)}:  For any character $\theta$ of $\Delta$, the Selmer groups $\Sel_i(A_f/\Q_\infty)^\theta$ are both cotorsion over $\Lambda$ for $i=1,2$\footnote{In \cite{HatleyLei}, the statement is about $\Sel_i(A_{f,s}/\Q_\infty)^\theta$ for any fixed integer $s$, which is equivalent to our hypothesis thanks to \cite[Remark 2.6]{HatleyLei}, noting that twisting by the cyclotomic character does not change $\Lambda$-cotorsion-ness of a module.}.
  

Various conditions guarantee these hypotheses, which include many cases.
 For the convenience of the reader, we recall these hypotheses and the appropriate references. 
\begin{center}
\begin{table}[h]
    \centering
    \begin{tabular}{|c||c|c|}
    \hline
      Hypothesis   &  Sufficient conditions & References \\
      \hline\hline
       \textbf{(irred)}  &\begin{tabular}{{@{}c@{}}}
       $p>k$ and \\$a_q(f)\not\equiv 1+q^{k-1} \pmod{\varpi}$ \\for some prime $q\equiv 1 \pmod{N}$ \\    \end{tabular}  & \begin{tabular}{{@{}c@{}}}\cite[Proof of Lem. 2.4]{DFG},\\ cf. \cite[Lem. 7.1]{HatleyLei}  \end{tabular}\\
       \hline
      \textbf{(inv)}  & $k\leq p$ &  \cite[Lem. 4.4]{lei11compositio}\\
       \hline
       \textbf{(tor)} & $a_p=0$ or $k \geq 3$& \cite[Start of Sec. 7, p. 1289]{HatleyLei}  \\
       \hline

    \end{tabular}
    
    \caption{\label{table}~}
\end{table} 
\end{center}

\begin{remark} The torsion assumption \textbf{(tor)} guarantees that $\Sel_i(A(s)/\Q_{\rm{cyc}})$ is $\Lambda$-cotorsion for all $s$. This follows from \cite[Remark 2.6]{HatleyLei}, and noting that twisting a finitely generated $\Lambda$-cotorsion module by the cyclotomic character results in a finitely generated $\Lambda$-cotorsion module. \end{remark}

In the references 
\cite[Thm. 7.3]{kobayashi03}, \cite[Prop. 6.4]{lei11compositio}, \cite[Thm. 6.5]{leiloefflerzerbes10}, \cite[Thm. 7.14]{sprung12}, one can find proofs that the conditions mentioned in the second column and the last row of Table 1 imply that the hypothesis \textbf{(tor)} is satisfied.


\subsection{The Bloch-Kato Selmer groups}\label{sec: BKloc} Let $V$ be a $\Q_p$-vector space defined with a $G_\Q$-action and $T$  be a $\Z_p$-stable lattice of $V$. 
Let us suppose that $\ell \neq p$. Recall that for $* \in \{V,V/T\}$ the unramified local cohomology group is defined as 
$$H^1_{\mathrm{ur}}(\Q_{\ell},*) :=\ker\big(H^1(\Q_\ell,*) \rightarrow H^1(I_\ell,*)\big),$$
where $I_\ell$ is the inertia subgroup at $\ell$.

Then the Bloch-Kato local conditions are defined as
\[ H^1_f(\Q_\ell, V)=\begin{cases} 
      H^1_{\mathrm{ur}}(\Q_\ell, V) & \ell \neq p, \\
      
      \ker\big(H^1(\Q_\ell,V) \rightarrow H^1(\Q_\ell,V \otimes \mathbb{B}_\mathrm{cris})\big) & \ell=p, 
   \end{cases}
\]
where $\mathbb{B}_\mathrm{cris}$ is the Fontaine's crystalline period ring. We also define $H^1_f(\Q_\ell,T)$ (resp. $H^1_f(\Q_\ell, V/T))$ as the preimage (resp. image) of $H^1_f(\Q_\ell,V)$ under the natural inclusion map (resp. projection map). 
In the $\ell \neq p$ case, we have the following  diagram (see e.g \cite[Section 3.1]{longovigni}) 
\begin{equation}
	\begin{tikzcd}
	0 \arrow[r] & H^1_{\mathrm{ur}}(\Q_\ell,V/T) \arrow[r]  & H^1(\Q_\ell,V/T) \arrow[r] &  H^1(I_\ell,V/T)  \\
	0 \arrow[r] & H^1_{\mathrm{ur}}(\Q_\ell,V)  \arrow[r] &  H^1(\Q_\ell,V) \arrow[u]  \arrow[r] &  H^1(I_\ell,V) \arrow[u],
	\end{tikzcd}
	\end{equation}
which shows that $H^1_f(\Q_\ell,V/T) \subset H^1_{\mathrm{ur}}(\Q_\ell,V/T)$. The index $$c_\ell(V/T):=[H^1_{\mathrm{ur}}(\Q_\ell,V/T):H^1_f(\Q_\ell,V/T)]$$
is finite (see \cite[Lemma 1.3.5]{Rubin2014} or \cite[Lemma 3.1]{VigniLongo}) and is defined as the \textit{$p$-part of the Tamagawa number} of $V/T$ at $\ell$.

Let $\Sigma$ be a finite set of places of $\Q$ containing the prime $p$, all the places that divide the level $N$, and the archimedean prime. Let $\Q_\Sigma$ be the maximal extension of $\Q$ unramified outside $\Sigma$. 
The Bloch-Kato Selmer group is defined as 
$$\Sel_{\BK}(A_{f,s}/\Q):=\ker\Big(H^1(\Q_\Sigma/\Q,A_{f,s}) \rightarrow \prod_{\ell \in \Sigma}H^1_{/f}(\Q_\ell,A_{f,s})\Big),$$
where 
$$H^1_{/f}(\Q_\ell,A_{f,s}):=\frac{H^1(\Q_\ell,A_{f,s})}{H^1_{f}(\Q_\ell,A_{f,s})}.$$

To conclude this section, we recall an important result of Greenberg.
\begin{lemma}[Greenberg's Lemma]\label{Greenlemma}
Let $M$ be a cofinitely generated, cotorsion module over the Iwasawa algebra $\OO[[T]]$, where $\OO$ is the ring of integers of any finite degree extension of $\Qp$. Let $f(T)$  be a generator of the characteristic ideal of the Pontryagin dual of $M$. Assume that $M^\Gamma$ is finite. Then $M_\Gamma$ is finite, $f(0) \neq 0$, and $$f(0) \sim |M^\Gamma|/|M_\Gamma|=\chi(M)$$
\end{lemma}
The quantity on the right is the \textit{Euler characteristic} of $M$, which will be made precise in the main section.
\begin{proof}When $\OO=\Zp$, this is \cite[Lemma 4.2]{greenberg89}. To make \cite[Proof of Lemma 4.2]{Greenberg} work for general $\OO$, notice that all you have to do is replace the first appearance of `$\Zp$' by `$\OO$.'\end{proof}

\section{Main result}\label{main:body}
Let $s$ be any integer and $\theta$ be a character of $\Delta$. The Pontryagin dual of $A_f(\kappa^s\theta^{-1})$ is $\T_f(\kappa^{-s}\theta)=\T_{f,s}(\theta\omega^s)$. Let $\A_{f,s}:=\T_{f,s} \otimes E/\OO_E=A_{f,k-s-2}$. We will write $$A(s):=\A_{f,s}(\theta \omega^s).$$ Note that we slightly differ from the notation of  \cite[Proof of Prop. 3.3, line -1 in page 1274]{HatleyLei} where it is denoted as $A$, since they fix $s$. 

We define the $\Gamma_0$-Euler characteristic of $\Sel_i(A(s)/\Q_{\rm{cyc}})$ as $$\chi(\Gamma_0,\Sel_i(A(s)/\Q_{\rm{cyc}}))=\frac{\#H^0(\Gamma_0,\Sel_i(A(s)/\Q_{\rm{cyc}}))}{\#H^1(\Gamma_0,\Sel_i(A(s)/\Q_{\rm{cyc}}))},$$ whenever the quantitites are finite.

When  $\Sel_i(A(s)/\Q_{\rm{cyc}})$ is $\Lambda$-cotorsion, let $f_{i,s} \in \Lambda$ be a  characteristic power series that generates its Pontryagin dual. 

\begin{theorem}\label{Thm:main1}
 Assume \textbf{(tor)}, \textbf{(inv)}, and \textbf{(irred)}. Choose $i\in \{1,2\}$. Then
for all but finitely many $s$, $\chi(\Gamma_0, \Sel_i(A(s)/\Q_{\rm{cyc}}))$ is well-defined. Moreover, $\Sel_{i}(A(s)/\Q)$ is finite, $f_{i,s}(0)\neq 0$, and  
\begin{equation}\label{eq: main0}
 \#\left(\frac{\OO_E}{f_{i,s}(0)\OO_E}\right)=\chi(\Gamma_0,\Sel_i(A(s)/\Q_{\rm{cyc}}))=\#\Sel_{i}(A(s)/\Q) \cdot  \prod_{\ell\in \Sigma, \\ \ell \nmid p} c_\ell(A(s)).   
\end{equation}

Further, suppose that $s \in \{0,...,k-2\}$. If $i=1$, assume also that $a_p(f) \neq \varepsilon(p)p^s +p^{k-s-2}$. Then $\Sel_i(A(s)/\Q)$ and $\Sel_{\BK}(A(s)/\Q)$ are finite, $f_{i,s}(0)\neq 0$, and  
\begin{equation}\label{eq: main1}
 \#\left(\frac{\OO_E}{f_{i,s}(0)\OO_E}\right)=\chi(\Gamma_0,\Sel_i(A(s)/\Q_{\rm{cyc}}))=\#\Sel_{\BK}(A(s)/\Q) \cdot  \prod_{\ell\in \Sigma, \\ \ell \nmid p} c_\ell(A(s)).   
\end{equation}

Here $c_\ell(A(s)):=[H^1_{\mathrm{ur}}(\Q_\ell,A(s)):H^1_f(\Q_\ell,A(s))]$ is the $p$-part of the \textit{Tamagawa number} of $A(s)$ at $\ell$. 
\end{theorem}


We record a proposition and two lemmas before proving the theorem. 

\begin{proposition}[Hatley--Lei]\label{prop:determins_s}
Under the hypotheses \textbf{(tor)},  {\textbf{(inv)}, and \textbf{(irred)}}, we have for each $i=1$ or $2$ that for all but finitely many integers $s$:\begin{enumerate}
    \item $\big(\Sel_i(A(s)/\Q_{\rm{cyc}}) \big)^{\Gamma_0}$  is finite,
    \item The { restriction map 
    $$\Sel_i(A(s)/\Q)  \rightarrow \big(\Sel_i(A(s)/\Q_{\rm{cyc}}) \big)^{\Gamma_0}$$ is injective with finite cokernel}.
\end{enumerate}
\end{proposition}


\begin{lemma}\label{a}
Let $\ell\in\Sigma$ and $v$ be any one of the finitely many places of $\Q_{\rm{cyc}}$ over $\ell$. \begin{enumerate}
    \item In the commutative diagram below, $h$ is an isomorphism and $\rho$ is surjective.
    \item In the case $\ell\neq p, \#\ker g_{/i,\ell} = c_\ell(A(s)).$
\end{enumerate}

	\begin{equation}\label{fundamental_diagram}
	\begin{tikzcd}
	0 \arrow[r] & \Sel_i(A(s)/\Q) \arrow[d, "a"] \arrow[r] &  H^1(\Q_\Sigma/\Q, A(s)) \arrow[d, "h"]  \arrow[r, "\rho"] & \underset{\ell \in \Sigma}{\prod} H^1_{/i}(\Q_\ell,A(s)) \arrow[d, "g=\prod g_{/i,\ell}"] \\
		0 \arrow[r] & \Sel_i(A(s)/\Q_{\rm{cyc}})^{\Gamma_0} \arrow[r] & H^1(\Q_\Sigma/\Q_{\rm{cyc}},A(s))^{\Gamma_0}  \arrow[r] & \underset{v\mid \ell, \ell \in \Sigma}{\prod} H^1_{/i}(\Q_{\cyc,v},A(s))^{\Gamma_0} 
	\end{tikzcd}
	\end{equation}
\end{lemma}

\begin{lemma}[Control Lemma at the prime $p$]\label{a2}
The vertical maps in the commutative diagram below are isomorphisms:
\begin{equation}\label{fundamental_diagram_p}
	\begin{tikzcd}
	0 \arrow[r] & H^1_i(\Q_p,A(s)) \arrow[d, "g_{i,p}"] \arrow[r] &  H^1(\Q_p, A(s)) \arrow[d, "g_p"]  \arrow[r,"c"] &  H^1_{/i}(\Q_p,A(s))  \arrow[d, "g_{/i,p}"]\\
	0 \arrow[r] & H^1_i(\Q_{\cyc,p},A(s))^{\Gamma_0} \arrow[r] & H^1(\Q_{\cyc,p},A(s))^{\Gamma_0}  \arrow[r] &  H^1_{/i}(\Q_{\cyc,p},A(s))^{\Gamma_0} 
	\end{tikzcd}
	\end{equation}
\end{lemma}

\begin{proof}[Proof of Theorem \ref{Thm:main1}]

\noindent The assumption \textbf{(tor)} gives us that the Euler characteristic $\chi(\Gamma_0, \Sel_i(A(s)/\Q_{\rm{cyc}}))$ is well-defined: Indeed, proposition \ref{prop:determins_s}(1) shows that $\Sel_i(A(s)/\Q_{\rm{cyc}})^{\Gamma_0}$ is finite. Hence by lemma \ref{Greenlemma}, we know that  $\Sel_i(A(s)/\Q_{\rm{cyc}})_{\Gamma_0}$ is finite. Since $\Gamma_0\cong \Z_p$, we can identify $H^1(\Gamma_0,\Sel_i(A(s)/\Q_{\rm{cyc}}))$ with $\Sel_i(A(s)/\Q_{\rm{cyc}})_{\Gamma_0}$ and so the Euler characteristic equals
$$\chi(\Gamma_0,\Sel_i(A(s)/\Q_{\rm{cyc}}))=\frac{\#\Sel_i(A(s)/\Q_{\rm{cyc}})^{\Gamma_0}}{\#\Sel_i(A(s)/\Q_{\rm{cyc}})_{\Gamma_0}}.$$

For the main part of the proof, choose $s$ so that $\Sel_i(A(s)/\Q_{\rm{cyc}})^{\Gamma_0}$ is finite. This only excludes finitely many $s$ by proposition \ref{prop:determins_s}. Applying the snake lemma to diagram \eqref{fundamental_diagram}, we obtain from lemma \ref{a}(1) that $\coker a$ is finite, $\ker a$ is trivial, and  $ \Sel_i(A(s)/\Q)$ is finite. From lemma \ref{a}(2) and lemma \ref{a2}, we then have $$\#\Sel_i(A(s)/\Q_{\rm{cyc}})^{\Gamma_0} =\#\Sel_i(A(s)/\Q) \cdot \prod_{\ell\in \Sigma, \\ \ell \nmid p} c_\ell(A(s)).$$

But the left-hand side is the Euler characteristic, since from \cite[Proof of Theorem 3.1, very end of Section 3]{HatleyLei}, $\Sel_i(A(s)/\Q_{\rm{cyc}})_{\Gamma_0}=0$, i.e. we really have

\begin{equation}\label{EC}
\chi(\Gamma_0, \Sel_i(A(s)/\Q_{\rm{cyc}})) = \#\Sel_i(A(s)/\Q) \cdot \prod_{\ell\in \Sigma, \\ \ell \nmid p} c_\ell(A(s)),
\end{equation}
which is the second equality in equation \eqref{eq: main0} of the theorem. 

Greenberg's Lemma \ref{Greenlemma} for $M=\Sel_i(A(s)/\Q_{\rm{cyc}})$ implies  $f_{i,s}(0)\neq 0$ and 
$$\#\left(\frac{\OO_E}{f_{i,s}(0)\OO_E}\right)=\chi(\Gamma_0,\Sel_i(A(s)/\Q_{\rm{cyc}})),$$

giving us the first equality.
(Note that this is the supersingular analogue of \cite[Prop. 4.3]{VigniLongo}).

To obtain the equalities in equation \eqref{eq: main1}, recall from \cite[Prop. 2.14 and Rem. 2.15]{HatleyLei} that for $i=2$ and $s \in \{0,...,k-2\}$, $$\Sel_2(A(s)/\Q)=\Sel_{\BK}(A(s)/\Q), $$ and that if $i=1$ and if we further assume that $a_p(f) \neq \varepsilon(p)p^s+p^{k-s-2}$, then $$\Sel_1(A(s)/\Q) \cong \Sel_{\BK}(A(s)/\Q).$$

\end{proof}

\begin{remark}
In the last part of the proof of Theorem \ref{Thm:main1}, we needed to restrict $s$ in the range $\{0,...,k-2\}$ because this is needed in the proof of \cite[Prop. 2.12]{HatleyLei}. In their proof they first use a result from \cite{LoefflerZerbes2014} where they need $s$ to be nonnegative. Later in their proof they use a result from  \cite{LLZ2017asymptotic} where they need $s$ to be in the range $\{0,...,k-2\}$ so that the twisted representation has appropriate Hodge-Tate weights\footnote{We thank Jeffrey Hatley for pointing this out to us.}. 
\end{remark}
\subsection{Proofs of Proposition \ref{prop:determins_s} , lemma \ref{a}, and lemma \ref{a2}}
\begin{proof}[Proof of Proposition \ref{prop:determins_s}]
{Consider the Selmer group $$\Sel_i(A(s)/\Q_{\rm{cyc}}) = \Sel_i(\A_{f,s}(\theta \omega^s)/\Q_{\rm{cyc}}) \cong \Sel_i(A_f(\theta \omega^{k-2})/\Q_{\rm{cyc}}) \otimes \kappa^{k-s-2}.$$
With $\eta=\theta^{-1}\omega^{2-k}$ (independent of $s$) we have the following isomorphism: 
$$\Sel_i(A_f(\theta \omega^{k-2})/\Q_{\rm{cyc}}) \otimes \kappa^{k-s-2} \cong \Sel_i(A_f/\Q_\infty)^\eta \otimes \kappa^{k-s-2}.$$
But the hypothesis \textbf{(tor)} gives that $\Sel_i(A_f/\Q_\infty)^\eta$ is $\Lambda$-cotorsion. Hence for all but finitely many $s \in \Z$, 
\begin{equation}\label{S1}
\big(\Sel_i(A_f/\Q_\infty)^\eta \otimes \kappa^{k-s-2}\big)^{\Gamma_0} \text{ is finite.}
\end{equation}
  This proves part (1) of Proposition \ref{prop:determins_s}. \vspace{.2cm}\\
In order to prove part (2), we need to  consider the fundamental diagram \eqref{fundamental_diagram}. Note that $\ker(h)$ and $\coker(h)$ can be regarded as inside the inflation-restriction exact sequence, and are zero because of hypothesis \textbf{(inv)}. (Cf. also \cite[Lemma 3.3]{HachimoriMatsuno}). Let $\ell \in \Sigma$ and $ v$ be any one of the finitely many places of $\Q_{\rm{cyc}}$ over $\ell$. By the inflation-restriction exact sequence, the kernel of the map $H^1(\Q_{\ell},A(s)) \rightarrow H^1(\Q_{\cyc,v},A(s))^{\Gamma_0}$ is $H^1(\Q_{\cyc,v}/\Q_{\ell},A(s)^{G_{\Q_{\cyc,v}}})$. This kernel is trivial if:
\begin{enumerate}
    \item $\ell=p$ (by \textbf{(inv)}), or if
    \item $\ell$ splits completely over $\Q_{\rm{cyc}}$.
\end{enumerate}
Assume that $\ell$ does not satisfy either of the above conditions. Let $\gamma_v$ be a topological generator for $\Gal(\Q_{\cyc,v}/\Q_{\ell})$, and let $B(s)=A(s)^{G_{\Q_{\cyc,v}}}.$ Then $$H^1(\Q_{\cyc,v}/\Q_{\ell}, A(s)^{G_{\Q_{\cyc,v}}}) \cong B(s)/(\gamma_v-1)B(s)$$ and we have an exact sequence 
$$0 \rightarrow A(s)^{G_{\Q_{\ell}}} \rightarrow B(s) \xrightarrow{\gamma_v-1} B(s) \rightarrow B(s)/(\gamma_v-1)B(s) \rightarrow 0.$$
Note that for all but finitely many $s$,
\begin{equation}\label{S2}
A(s)^{G_{\Q_{\ell}}} \text{ is finite}.
\end{equation}
  In this case, we have that $B(s)_{div} \subset (\gamma_v-1)B(s)$, where $B(s)_{div}$ is the maximal divisible subgroup of $B(s)$. Hence it follows that $B(s)/(\gamma_v-1)B(s)$ is bounded by $B(s)/B(s)_{div}$, and this is finite. Hence for all but finitely many $s\in \Z$, the proof of Proposition \ref{prop:determins_s} is  complete.
}
\end{proof}

\begin{proof}[Proof of Lemma \ref{a}]
{As noted above in the proof of Proposition \ref{prop:determins_s}, the fact that the map $h$ is an isomorphism follows form the inflation-restriction exact sequence and the hypothesis \textbf{(inv)}.}

The surjectivity of $\rho$ follows from \cite[Prop. 3.3]{HatleyLei} and its proof.

The Greenberg local condition at $v$ (see \cite[Sec. 3.4.2]{VigniLongo})  coincides with the Bloch--Kato local condition defined in section \ref{sec: BKloc}. Hence, by \cite[Lemma 5.3]{VigniLongo},  we obtain $$\#\ker g_{/i,\ell} = c_\ell(A(s)),$$
which is indeed finite by \cite[Lemma 3.1]{VigniLongo}. 
(It is also interesting to note that if $v \notin \Sigma$ the Tamagawa factors are trivial (see \cite[Lemma 3.3]{VigniLongo})).
\end{proof}

\begin{proof}[Proof of Lemma \ref{a2}]
The kernel  of the map $g_p$ is $H^1(\Q_{\cyc,p}/\Q_p,A(s)^{G_{\Q_{\cyc,p}}})$  and its cokernel maps into
$H^2(\Q_{\cyc,p}/\Q_p,A(s)^{G_{K_{\infty,p}}})$ which are both trivial because of our hypothesis \textbf{(inv)}. Hence the map $g_p$ is an isomorphism. 

Next, we analyze the kernel and the cokernel of the map  $g_{/i,p}$.

We have the following isomorphism by duality
$$H^1_i(\Q_{\cyc,p},A(s)) =\big(\mathrm{Im}(\col_{f,i})^\theta \otimes \kappa^{-s}\big)^\vee.$$
Since $\mathrm{Im}(\col_{f,i})^\theta \subset \Lambda$, it implies that $$\big(\mathrm{Im}(\col_{f,i})^\theta \otimes \kappa^{-s}\big)^{\Gamma_0}=0.$$
Hence we have
\begin{equation}\label{eq:explain}
H^1_i(\Q_{\cyc,p},A(s))_{\Gamma_0}=0.
\end{equation}
Consider the short exact sequence 
$$0 \rightarrow H^1_i(\Q_{\cyc,p},A(s)) \rightarrow H^1(\Q_{\cyc,p}, A(s)) \rightarrow H^1_{/i}(\Q_{\cyc,p}, A(s)) \rightarrow 0.$$
Combining \eqref{eq:explain} with the isomorphism 
$$H^1_i(\Q_{\cyc,p},A(s))_{\Gamma_0}\cong H^1(\Gamma_0, H^1_i(\Q_{\cyc,p},A(s))),$$
(see e.g. \cite[Prop. 1.7.7]{Cohonumbfields}), we obtain $$H^1_{/i}(\Q_{\cyc,p},A(s))^{\Gamma_0}\cong \frac{H^1(\Q_{\cyc,p},A(s))^{\Gamma_0}}{H^1_i(\Q_{\cyc,p},A(s))^{\Gamma_0}},$$
 as in p. 1276 of \cite{HatleyLei}.

But it follows from \cite[p. 1276, line 10]{HatleyLei} that
$$\frac{H^1(\Q_{\cyc,p},A(s))^{\Gamma_0}}{H^1_i(\Q_{\cyc,p},A(s))^{\Gamma_0}} \cong H^1_{/i}(\Q_p,A(s))$$
using the Hochschild-Serre spectral sequence and the facts that
\begin{align*}
  H^1(\Q_{\cyc,p},A(s))^{\Gal(\Q_{\cyc,p}/\Q_p)}&=H^1(\Q_p,A(s)), \\
  H^1_i(\Q_{\cyc,p},A(s))^{\Gal(\Q_{\cyc,p}/\Q_p)}&=H^1_i(\Q_p,A(s)),
\end{align*}
cf. \cite[Remark 2.5]{HatleyLei}.
Finally, we obtain from this that $$H^1_{/i}(\Q_{\cyc,p},A(s))^{\Gamma_0}\cong  H^1_{/i}(\Q_p,A(s)),$$
which shows that the map $g_{/i,p}$ is also an isomorphism.

The map $c$ is surjective by definition, and hence by the snake lemma we obtain that the map $g_{i,p}$ is also an isomorphism. (This generalizes an analogous property that has been proved by Ahmed and Lim in the case of elliptic curves (see \cite[Lemma 2.6]{AhmedLim})).
\end{proof}
{
    To summarize, the local conditions at the prime $p$ for the two signed Selmer groups $\Sel_i(A(s)/\Q)$ and $\Sel_i(A(s)/\Q_\cyc)^{\Gamma_0}$ are shown to be isomorphic in Lemma \ref{a2}, which allows us to bypass computing the kernel of $g_{/i,p}$; this is the key ingredient in computing the Euler characteristic formula in Theorem \ref{Thm:main1}. The signed conditions here play a crucial role because the Selmer groups are defined such that their local conditions at $p$ satisfy the compatibility conditions \eqref{Comp_local_p} and \eqref{cyc_loc_p}.
    Also note that we have chosen our twists $A(s)$ (i.e. those $s$ for which both \eqref{S1} and \eqref{S2}  are true) such that Proposition \ref{prop:determins_s} holds; part (1) of this proposition is used to guarantee that the Euler characteristic is well-defined by Greenberg's Lemma (cf. Lemma \ref{Greenlemma}). To treat the places away from $p$, we rely on the methods of Longi--Vigni \cite{VigniLongo}.
}

\bibliographystyle{alpha}
\bibliography{main}
\end{document}